\documentclass[conference]{IEEEtran}
\IEEEoverridecommandlockouts
\usepackage{cite}
\usepackage{amsmath,amssymb,amsfonts}
\usepackage{algorithmic}
\usepackage{graphicx}
\usepackage{textcomp}
\usepackage{xcolor}
\usepackage{dsfont}
\usepackage{amsthm}
\usepackage{enumitem}
\usepackage{subcaption}
\usepackage[font=small]{caption}

\usepackage{geometry}
\newtheorem{theorem}{Theorem}

\newtheorem{lemma}{Lemma}
\newtheorem{remark}{Remark}

\newtheorem{definition}{Definition}

\newcommand{\Na}{{N_{\rm a}}}

\def\BibTeX{{\rm B\kern-.05em{\sc i\kern-.025em b}\kern-.08em
    T\kern-.1667em\lower.7ex\hbox{E}\kern-.125emX}}
\begin{document}

\title{Patterns of Nonlinear Opinion Formation  \\on Networks
\thanks{Supported by NSF grant CMMI-1635056, ONR grant N00014-19-1-2556, ARO grant W911NF-18-1-0325, DGAPA-UNAM PAPIIT grant IN102420, and Conacyt grant A1-S-10610, and 
NSF 
Grad. Res. Fell. DGE-2039656.
}
}

\author{Anastasia~Bizyaeva, Ayanna Matthews,
        Alessio~Franci,
        and~Naomi~Ehrich~Leonard
\thanks{A. Bizyaeva, A. Matthews and N.E. Leonard are with the Department
of Mechanical and Aerospace Engineering, Princeton University, Princeton,
NJ, 08544 USA; e-mail: bizyaeva@princeton.edu, naomi@princeton.edu.}
\thanks{A. Franci is with the National Autonomous University of Mexico, 04510 Mexico City, Mexico. e-mail: afranci@ciencias.unam.mx}}

\maketitle

\begin{abstract}
When communicating agents form opinions about a set of possible options, agreement and disagreement are both possible outcomes. Depending on the context, either can be desirable or undesirable. We show that for nonlinear opinion dynamics on networks, and a variety of network structures, the spectral properties of the underlying adjacency matrix fully characterize the occurrence of either agreement or disagreement. We further show how the corresponding eigenvector centrality, as well as any symmetry in the network, informs the resulting patterns of opinion formation and agent sensitivity to input that triggers opinion cascades.
\end{abstract}


\section{Introduction}


Multi-agent systems that perform distributed control tasks in uncertain or dynamic contexts benefit when agents use network communications to form and change opinions about context-dependent options.
For example, network communications can help autonomous multi-robot teams 
make better navigational choices among alternative routes or better allocation choices among alternative tasks.
Mathematical models of opinion dynamics over networks are often used to bring a group to a desired opinion configuration. In a task-allocation context, agreement is not necessarily the only desirable opinion configuration, as sometimes agents are better off  exploring different routes or doing different tasks.

A general model of opinion dynamics for distributed agents on a network was recently introduced in \cite{Bizyaeva2020,Franci2020}. In this multi-agent multi-option model, real-valued opinions evolve in continuous time according to a nonlinear update rule that saturates network exchanges. A key feature of the model is the emergence of consensus  and dissensus as equilibrium opinion configurations, even when agents are homogeneous, receive no input, and communicate over an all-to-all network.
The emergence of consensus and dissensus depends on a small number of parameters that distinguish the interactions between agents as cooperative or competitive. 

In \cite{Bizyaeva2020} the behavior of the general model is examined with particular attention to all-to-all and vertex-transitive cycle network topologies. The work identifies parameter regimes that correspond to consensus and dissensus (a heterogeneous opinion configuration with neutral average opinion).   In the present paper we take first steps to 
analyze opinion formation on other classes of networks. We consider opinion formation for two options and homogeneous agents that communicate over networks 
that include $K$-regular, bipartite, and strongly connected directed graphs. We examine how network structure influences the group 
outcome of the opinion formation process, and we prove that, generically, agreement and disagreement arise on these networks. We show that the parameter regimes associated with consensus and dissensus for complete graphs in \cite{Bizyaeva2020}  correspond precisely to agreement and disagreement regimes for more general networks. 

The engineering literature on distributed opinion dynamics typically associates formation of opinions in continuous time to the spectral properties of the Laplacian matrix of the network graph. In linear consensus protocols, the governing equations take on the structure of the graph Laplacian, and consensus is achieved as 
opinions converge to its kernel \cite{OlfatiSaber2004}. Analogous distributed Laplacian schemes are considered with antagonistic (signed) interconnections in, e.g., \cite{Altafini2013,Liu2017,shi2019dynamics}. When graphs are structurally balanced, i.e., when the signed graph Laplacian has a zero eigenvalue \cite{Altafini2013}, 
such schemes give rise to clustered disagreement on the network. A nonlinear model of distributed consensus formation with saturated network interactions studied in \cite{AF-VS-NEL:15a,fontan2017multiequilibria,Gray2018,Abara2018,Fontan2018} also relies on Laplacian-like structure of the governing equations, with each agent weighting its opinion state based on its in-degree on the network graph. Linearization of this model about the unopinionated state 
yields linear Laplacian consensus dynamics.

We study opinion dynamics of homogeneous agents with governing equations that do not necessarily have a Laplacian structure. We show how spectral properties of the {\em adjacency matrix} of the underlying graph, rather than those of its Laplacian, play a key role in characterizing the opinion formation process. In \cite{fontan2017multiequilibria} the spectral properties of the 
adjacency matrix are tied to a necessary and sufficient condition for existence of multi-stable consensus equilibria. However, unlike our model, 
the model in \cite{fontan2017multiequilibria} relies on a Laplacian structure, and so the 
dynamics do not inherit the full symmetry properties of the network graph. 
For our model, we show that the largest and smallest eigenvalues of the adjacency matrix determine if nonzero opinions form, and the associated eigenspaces select the sign and relative magnitude of resultant opinions on the network. With the addition of dynamics on an attention parameter, the eigenvectors also 
determine which agents in the group are maximally sensitive to 
inputs. We illustrate how this 
can be used to 
trigger opinion cascades. 

In Section \ref{sec:model} we present the opinion dynamics model of \cite{Bizyaeva2020} on a network in the case of homogeneous agents and two options.
In Section \ref{sec:agree_disagree} we prove that agreement and disagreement equilibria, of which consensus and dissensus are special cases, arise on networks as 
bifurcations from an unopinionated equilibrium.  We prove relationships between the magnitude of an agent's equilibrium opinion and its centrality on the network in Sections \ref{sec:patterns}, and to the symmetry properties of the network graph in Section~\ref{sec:symmetry}. In Section \ref{sec:sensitivity} we introduce attention feedback dynamics and demonstrate how the group's sensitivity to input relates to the adjacency eigenvectors. We conclude in Section \ref{sec:dis} . 

\section{Opinion Dynamics Model\label{sec:model}}

We study a nonlinear model of $N_a$ homogeneous agents forming opinions about two options, a specialization of the general heterogeneous multi-agent, multi-option,  model introduced in \cite{Bizyaeva2020,Franci2020}. The {\em opinion of each agent $i$} is captured by a real-valued variable $x_{i} \in \mathds{R}$. 
When $x_{i} = 0$ agent $i$ has a \textit{neutral opinion}, and when $x_{i} > 0$ ($< 0$) agent $i$ \textit{favors} option A (option B). A greater magnitude $|x_{i}|$ corresponds to a stronger commitment of agent $i$ to one of the options.  We define a threshold $\vartheta > 0$ to formalize what is meant by ``close''; the value of  $\vartheta$ can be chosen for purposes of interpretation. 
We call an agent \textit{opinionated} when $|x_{i}| > \vartheta > 0$, and \textit{unopinionated} otherwise.

The {\em directed} graph $\bar G = (V, \bar E)$ encodes which agents can communicate with which other agents. The set of vertices $V = \{ 1, \dots, N_{a} \}$ represents the set of $N_{a}$ agents, and  edges $\bar E \subseteq V \times V$ represent interactions between agents. If edge $\bar e_{ik} \in \bar E$, then agent $k$ is a \textit{neighbor} of agent $i$. 
When inter-agent communication  is bidirectional, $G$ is  {\em undirected}: if $\bar e_{ik} \in \bar E$, then $\bar e_{ki} \in \bar E$.  The graph adjacency matrix $\bar A \in \mathds{R}^{N_{a} \times N_{a}}$ encodes interaction weights with element $\bar a_{ik} \neq 0$ if and only if $\bar e_{ik} \in \bar E$. 
The sign of 
inter-agent weights determines whether 
agents are \textit{cooperative} ($\bar a_{ik} > 0$) or \textit{competitive} ($\bar a_{ik} < 0$). In this paper we specialize $\bar a_{ik}  \in \{0,1\}$. 

Let $\mathbf{x} = (x_{1}, \dots, x_{N_{a}}) \in \mathds{R}^{N_{a}}$ be the {\em opinion state of the group}. When $\mathbf{x} = \mathbf{0}$, the group is in the {\em neutral state}. When all agents are unopinionated, the group is in an {\em unopinionated state}.  A pair of opinionated agents $i,k$ \textit{agree} if both  share the same qualitative opinion state, i.e., $\operatorname{sign}(x_{i}) = \operatorname{sign}(x_{k})$. When all agents agree, the group is an {\em agreement state}. The {\em consensus state} is a special type of agreement state in which 
opinions are close in value, i.e. $|x_{i} - x_{k}| < \vartheta$ for all $i,k \in V$. A pair of opinionated agents \textit{disagrees} if each has a different qualitative opinion state. If at least one pair of agents disagrees, then the group is in a \textit{disagreement state}. The \textit{dissensus} state is a special type of disagreement state in which individual agents may be opinionated but the group is unopinionated on average, i.e., $\left|\sum_{i = 1}^{N_{a}} x_{i}\right| < \vartheta N_a$.

Each agent's opinion is updated in continuous time as a function of three key terms:  a linear damping term, a nonlinear network interaction term that includes self-reinforcement, and an additive input term:
\begin{equation}
\dot{x}_{i} = - d x_{i} +  u_{i} S\left(\alpha x_{i} + \gamma \sum_{\substack{k=1 \\ k \neq i}}^\Na \bar{a}_{ik} x_{k}\right) + b_{i} := h_{i}(\mathbf{x}). \label{eq:opinion_dynamics_hom}
\end{equation}
$S: \mathds{R} \to \mathds{R}$ is an odd saturating function satisfying $S(0) = 0$, $S'(0) = 1$, $\operatorname{sign} \left( S''(z) \right) = - \operatorname{sign}(z)$. The saturation function is applied to the sum of two terms.  The first is a self-reinforcement term with 
weight $\alpha \geq 0$.  The second is the sum of opinions of agent $i$'s neighbors with weight $\gamma \in \mathds{R}$. That the network interactions are saturated using $S$ in \eqref{eq:opinion_dynamics_hom} means the opinion dynamics of agent $i$ are proportionally sensitive to changes in these interactions when the corresponding opinion magnitudes are small, but that the influence of these interactions levels off when the corresponding opinion magnitudes become large.
In all simulations we let  $S = \tanh$; however the results hold qualitatively for more general sigmoidal functions.  
For the general model, motivation, and relation to the broader literature see \cite{Bizyaeva2020,Franci2020}.

The parameter $d>0$ is the 
\textit{damping coefficient}. In the absence of network interactions, agent $i$'s opinion $x_i$ will  converge exponentially to 
$b_{i}/d$ at a rate determined by $d$. The control parameter $u_{i}>0$ is the \textit{attention} of agent $i$ to network interactions. Attention $u_i$ governs the 
influence of the saturated network interactions, relative to damping,  on the opinion dynamics of agent $i$. When $u_i$ is above a critical threshold that grows with $d$, the magnitude of $x_i$ 
grows nonlinearly to a value much larger than $b_i/d$. Attention $u_i$ can be fixed, time-varying, or defined to evolve according to  state-dependent (closed-loop) dynamics. The  \textit{input} $b_{i} \in \mathds{R}$ is an external signal or internal bias that stimulates an agent 
 in favor of option A (option B) when $b_{i} > 0$ ($<0$). For most of 
 this paper we focus on the zero-input case: $b_i = 0$ for all $i$.  We consider  nonzero input in Section \ref{sec:sensitivity}.

Let $W(\lambda_i)$ be the generalized eigenspace of $\bar A$ relative to its eigenvalue~$\lambda_i$. Let $\lambda_{max}$ 
and $\lambda_{min}$ be the eigenvalues of $\bar{A}$ with largest and smallest real parts. In the following lemma we state several useful properties of $\lambda_{min}$, $\lambda_{max}$, and their associated eigenspaces.

\begin{lemma}

\textbf{A.} When $\bar{G}$ is strongly connected (directed), $\lambda_{max} > 0$ is real with multiplicity 1, and for any nonzero vector $\mathbf{v} = (v_{1}, \dots, v_{N_{a}}) \in W(\lambda_{max})$, $v_{i} \neq 0$ and  $\operatorname{sign}(v_{i}) = \operatorname{sign}(v_{k})$ for all $i,k \in V$;

\textbf{B.} When $\bar{G}$ is connected (undirected),  $\lambda_{min} < 0$ is real and for any nonzero vector $\mathbf{w} = (w_{1}, \dots, w_{N_{a}}) \in W(\lambda_{min})$, $\operatorname{sign}(v_{i}) = - \operatorname{sign}(v_{k})$ for at least one pair of $i,k \in V$. \label{lem:lambdas}
\end{lemma}
\begin{proof}Observe that ${\rm tr\,}\bar{A}=0$  so  ${\rm Re\,}(\lambda_{min})<0$. When $\bar{G}$ is undirected, $\bar{A}$ is symmetric and $\lambda_{min}$ is real.
Further, $\bar{A}$ is a nonnegative matrix. Parts A and B follow from the Perron-Frobenius Theorem \cite[Theorem 11]{farina2011positive}. 
\end{proof}


\section{Agreement and Disagreement\label{sec:agree_disagree}}

In this section we study opinion dynamics (\ref{eq:opinion_dynamics_hom}) with  static $u_{i} := u \geq 0$ for all $i\in V$ and show how cooperative agents ($\gamma > 0$) give rise to agreement, whereas competitive agents ($\gamma < 0$) give rise to disagreement. In the following theorem, we expand upon the result in \cite[Theorem IV.1]{Bizyaeva2020} for two-option networks, and describe the steady-state solutions that arise from  \eqref{eq:opinion_dynamics_hom} in different parameter regimes.
\begin{theorem}
The following hold true for \eqref{eq:opinion_dynamics_hom} with $u_{i} := u \geq 0$ and $b_{i} = 0$ for all $i = 1, \dots, N_{a}$: \\
\textbf{A. Cooperation leads to agreement:} Let $\bar{G}$ be a strongly connected directed graph. If $\gamma > 0 $, the neutral state $\mathbf{x} = 0$ is a locally exponentially stable equilibrium for $0 < u < u_{a}$ and unstable for $u >u_{a}$, with
    \begin{equation}
        u_{a} = \frac{d}{\alpha + \gamma \lambda_{max}}. \label{eq:u_agree}
    \end{equation}
    At $u = u_{a}$, branches of agreement equilibria, $x_{i} \neq 0$, $\operatorname{sign}(x_{i}) = \operatorname{sign}(x_{k})$ for all $i,k \in V$, emerge in a steady-state bifurcation off  of $\mathbf{x} = \mathbf{0}$
    along $W(\lambda_{max})$; \\
\textbf{B. Competition leads to disagreement:} Let $\bar{G}$ be a connected undirected graph. If $ \gamma < 0 $ the neutral state $\mathbf{x} = \mathbf{0}$ is a locally exponentially stable equilibrium for $0 < u < u_{d}$ and unstable for $u >u_{d}$, with
    \begin{equation}
        u_{d} = \frac{d}{\alpha +\gamma \lambda_{min}} . \label{eq:u_disagree}
    \end{equation}
    At $u = u_{d}$, branches of disagreement equilibria, $\operatorname{sign}(x_{i}) = - \operatorname{sign}(x_{k})$ for at least one pair $i,k \in V $, $i \neq k$, emerge in a steady-state bifurcation off  of $\mathbf{x} = \mathbf{0}$ 
    along $W(\lambda_{min})$.
\label{thm:agree_disagree}
\end{theorem}

\begin{proof}
The Jacobian of \eqref{eq:opinion_dynamics_hom} evaluated at 
$\mathbf{x} = \mathbf{0}$ is $J(\mathbf{0}) = (u \alpha - d) \mathcal{I} + u \gamma \bar{A}$,
where $\mathcal{I}$ 
is the identity matrix. The eigenvalues of 
$J(\mathbf{0})$ are $ \mu_{i} = u (\alpha + \gamma \lambda_{i}) - d $ where $\lambda_{i}$ is an eigenvalue of $\bar{A}$. When $0 \leq u < \min_{i} \frac{d}{\alpha + \gamma {\rm Re\,}(\lambda_{i})}$, ${\rm Re\,}(\mu_{i})$ is negative for all $i \in V$ and $\mathbf{x} = \mathbf{0}$ 
is locally exponentially stable. For values of $u$ above this bound the origin is unstable. When $\gamma > 0$ this bound  $u_a$ corresponds to $\lambda_{i} = \lambda_{max}$ and 
is given by (\ref{eq:u_agree}). When $\gamma < 0$ the bound $u_d$ corresponds to $\lambda_{i} = \lambda_{min}$ and 
is given by (\ref{eq:u_disagree}). Thus, when $\gamma>0$ ($\gamma<0$) a steady-state bifurcation happens at $u=u_a$ ($u=u_d$) along $W(\lambda_{max})$ ($W(\lambda_{min})$). The rest follows from Lemma \ref{lem:lambdas}.
\end{proof}


\begin{remark}
Due to space constraints, we defer to a future publication detailed analysis of stability of the agreement and disagreement equilibria emerging at agreement and disagreement opinion-forming bifurcations. For one-dimensional kernels, 
we expect these equilibria to be generically stable, as easily verifiable using center-manifold reduction arguments. 
\end{remark}
 
A main takeaway of  Theorem \ref{thm:agree_disagree} is that the spectral properties of $\bar{A}$ inform the opinion formation outcomes on the network. Characterizing the eigenvalues $\lambda_{min}$, $\lambda_{max}$ along with their associated eigenspaces $W(\lambda_{min})$, $W(\lambda_{max})$ is equivalent to characterizing the primary branches of opinionated steady-state solutions of~\eqref{eq:opinion_dynamics_hom} emerging at bifurcations from the neutral state. We use this approach to systematically classify the opinion patterns that arise for various networks with spectral properties that are well known or easily established. For larger and less structured networks these quantities can be easily computed numerically. Figures \ref{fig:GraphClasses} and \ref{fig:MoreGraphs} illustrate agreement and disagreement equilibria of \eqref{eq:opinion_dynamics_hom} on a variety of graphs, with 
$u$ 
slightly above the 
bifurcation point. 

\begin{figure}
    \centering
    \includegraphics[width = 0.47\textwidth]{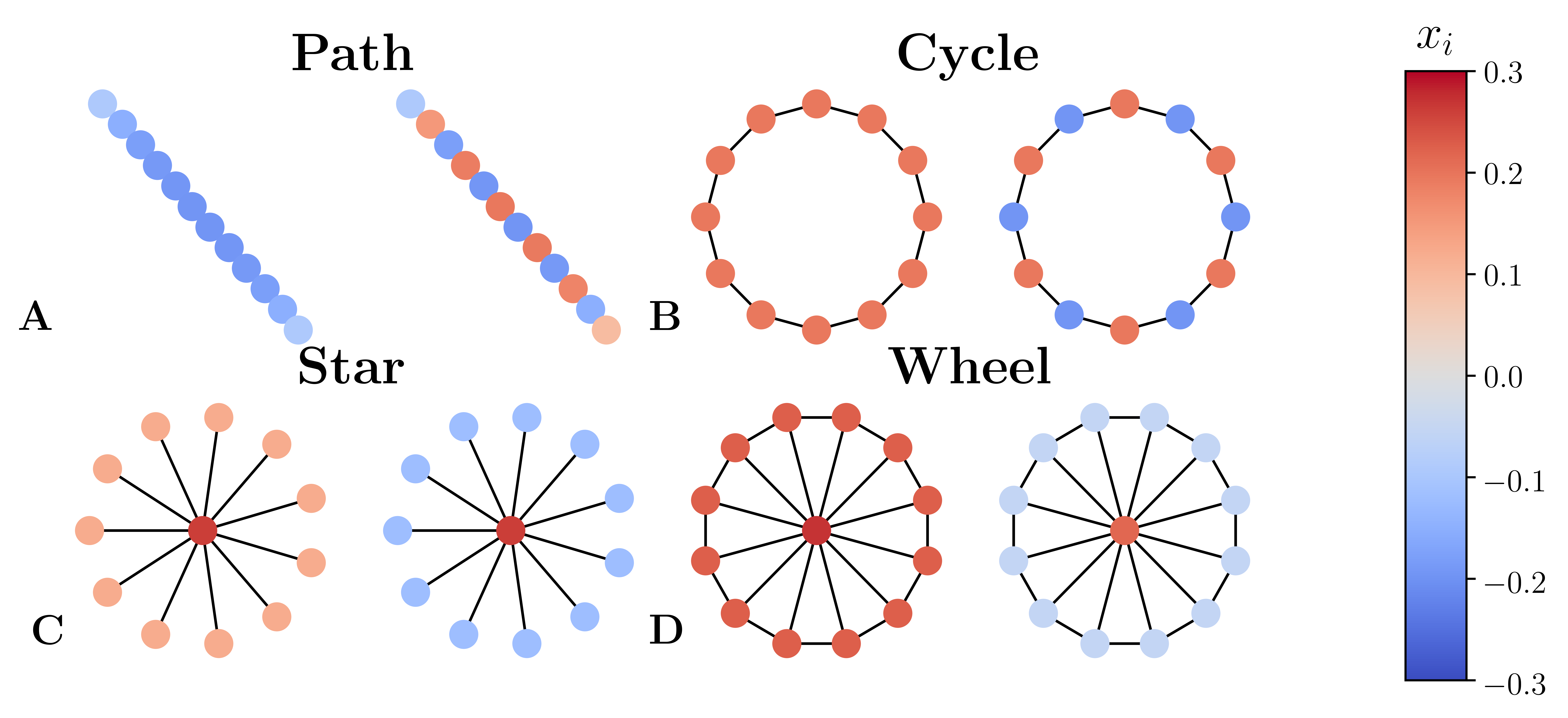}
    \caption{Steady-state patterns of agreement with $\gamma = 1.3$  (left) and disagreement with $\gamma = -1.3$  (right)  from simulation of opinion dynamics \eqref{eq:opinion_dynamics_hom} for four undirected graph types. 
    Node color represents opinion $x_{i}$ 
    at $t = 500$. All nodes have $b_{i} = 0$ and randomized initial opinions drawn from 
    $U(-1,1)$. 
    Parameters: $d = 1$, $\alpha = 1.2$, 
    $u = 0.31$ for path and cycle, $u = 0.26$ for star and wheel.}
    \label{fig:GraphClasses} \vspace*{-5mm}
\end{figure} 
 
Network consensus and dissensus are special cases of agreement and disagreement described in Theorem \ref{thm:agree_disagree}. 
Let the {\it consensus space} be  $W_{c} := \operatorname{span}\{ \mathbf{1} \}$ and the {\it dissensus space} be $W_{d} := \{ \mathbf{x} \in \mathds{R}^{N_{a}} : \sum_{i = 1}^{N_{a}} x_{i} = 0 \}$. Observe that $W_c$ and $W_d$ are orthogonal complements in $\mathds{R}^{N_{a}}$. Outside a $\vartheta$-neighborhood of the origin, the $\vartheta$-neighborhood of $W_c$ is made of consensus solutions, whereas the $\vartheta$-neighborhood of $W_d$ is made of dissensus solutions. 
We show next that for 
graphs in which all agents have the same number of neighbors, the agreement and disagreement equilibria of Theorem \ref{thm:agree_disagree} correspond to 
consensus and dissensus.

\begin{definition}
$\hat{G} = (\hat{V}, \hat{E})$ is a {\em $K$-regular graph} 
if every vertex $i \in \hat{V}$ has exactly $K$ neighbors.
\label{def:reg}
\end{definition}

\begin{lemma}
If $\bar{G}$ is undirected, connected, and $K$-regular, all vectors $\mathbf{x} = (x_{1}, \dots, x_{N_{a}}) \in W(\lambda_{min})$ satisfy $ \sum_{i = 1}^{N_{a}} x_{i} = 0$. \label{lem:Kregular}
\end{lemma}
\begin{proof}
Observe that for a connected, $K$-regular graph, $\lambda_{max}=K$ and $W(\lambda_{max}) = W_{c}$.  Because $\bar A$ is symmetric, all its generalized eigenspaces are orthogonal and, thus, $W(\lambda_{min}) \subseteq W_{d}$. The statement of the lemma follows.
\end{proof}

\begin{theorem}[Consensus and Dissensus]
 If $\bar{G}$ is undirected, connected, and $K$-regular, the agreement bifurcations at $u = u_{a}$ with $\gamma > 0$ give rise to consensus solutions $|x_{i} - x_{k}|<\vartheta$ for all $i,k \in V$, and the disagreement bifurcations at $u = u_{d}$ with $\gamma < 0$ give rise to dissensus solutions 
 $\left|\sum_{i = 1}^{N_{a}} x_{i}\right| < \vartheta N_a$. \label{thm:consensus_dissensus}
\end{theorem}
\begin{proof}
From Theorem \ref{thm:agree_disagree} and Lemma \ref{lem:Kregular},  an opinion-forming bifurcation emerges along  the consensus space for $\gamma>0$ and dissensus space for $\gamma<0$. This means that the resulting equilibrium solutions are arbitrarily close to the consensus (dissensus) space as $u\searrow u_a$ ($u\searrow u_d$), where the notation $\searrow$ signifies approaching in value from above.
\end{proof}

Figure \ref{fig:GraphClasses}B and Figure \ref{fig:MoreGraphs}A,B are examples of consensus and dissensus on 2-regular graphs (a cycle with an even and odd number of nodes) and a 3-regular graph.

\begin{figure}
    \centering
    \includegraphics[width = 0.47\textwidth]{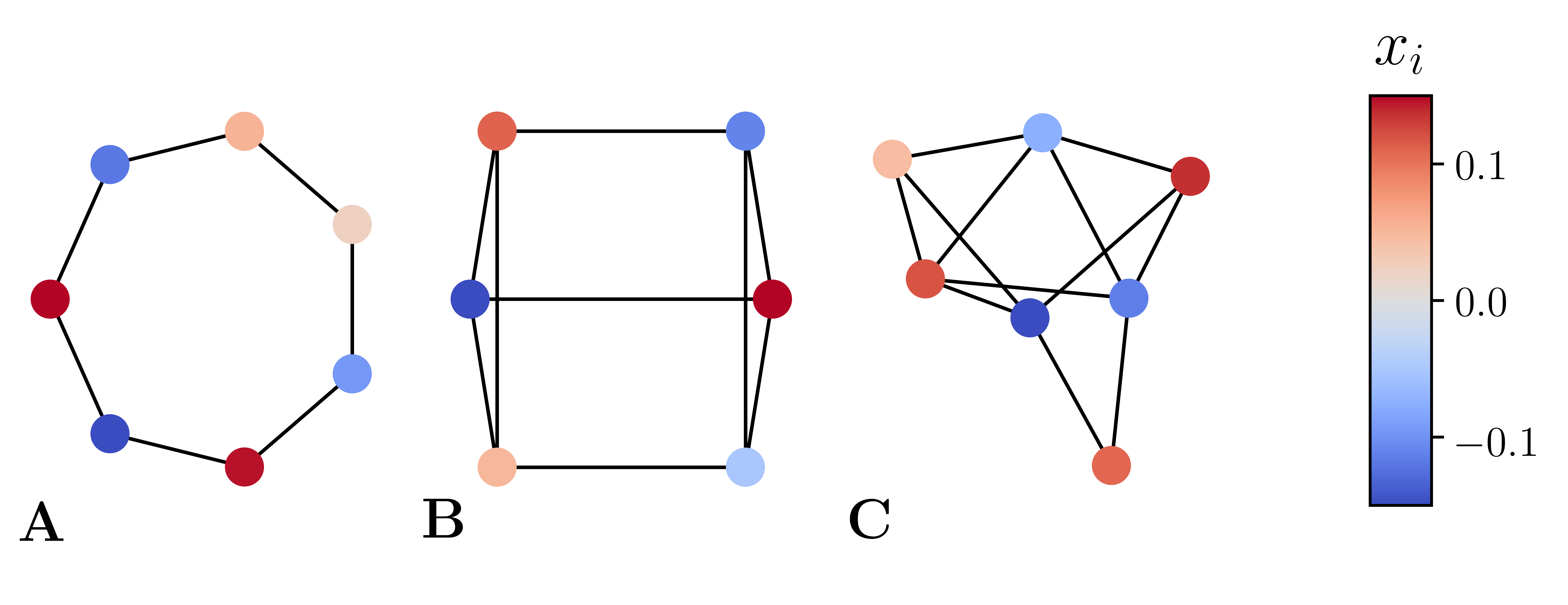}
    \caption{Disagreement patterns on odd cycle (A), 3-regular (B) and  randomly generated (C) graphs.  
    Parameters are $d = 1$, $\alpha = 0.5$, $\gamma = -0.5$, $u = u_{d} + 0.01$. All else is as in Fig.~\ref{fig:GraphClasses}.}
    \label{fig:MoreGraphs} \vspace*{-5mm}
\end{figure}


\section{Agent Centrality \label{sec:patterns}}

In this section we examine how equilibrium opinions of agents depend on their location in the graph. We show that at an opinion-forming bifurcation, an agent's opinion strength is often determined by its relative location in the network as quantified by a suitable centrality measure. 

A \textit{centrality measure}  ranks how central each node is in a network, i.e., measures its influence over some emergent network property. We recall the definition of a well known network centrality measure originally proposed in \cite{bonacich1972}:

\begin{definition}[Eigenvector Centrality]
The entries of the normalized positive left eigenvector corresponding to the eigenvalue $\lambda_{max}$ of $\hat{A}$ for a (directed or undirected) graph $\hat{G}$ provide a centrality measure for the nodes of the graph.\label{def:centrality}
\end{definition}
It is shown in \cite{bonacich2007some} that eigenvector centrality, deriving from the adjacency matrix and not the Laplacian, is particularly useful for graphs on which agents with high in-degree have many neighbors of low in-degree, and is related to several other common graph centrality measures. Let $\mathbf{v}^{c} = (v^{c}_{1}, \dots, v^{c}_{N_{a}})$ be the centrality eigenvector for $\bar{A}$ from Definition~\ref{def:centrality}. For undirected graphs, this centrality vector determines the opinion strength of each agent at an agreement equilibrium predicted by Theorem~\ref{thm:agree_disagree}A, i.e., the larger an agent's centrality, the stronger its opinion at  agreement. 

\begin{theorem} Consider opinion dynamics \eqref{eq:opinion_dynamics_hom} with undirected, connected $\bar{G}$, $u_{i} := u \geq 0$, $b_{i} = 0$, and $\alpha > 0$. Agreement equilibria $\mathbf{x} = (x_{1}, \dots, x_{N_{a}})$ described in Theorem \ref{thm:agree_disagree}.A satisfy $|x_{i}| < |x_{k}|$ if $v^{c}_{i} < v^{c}_{k}$ and $|x_{i}| = |x_{k}|$ if $v^{c}_{i} = v^{c}_{k}$ for all $i,k = 1, \dots, N_{a}$. \label{thm:centrality_agr}
\end{theorem} 
\begin{proof}
Let $\mathbf{v} = (v_{1}, \dots, v_{N_{a}})$ be the normalized right $\lambda_{max}$-eigenvector of $\bar{A}$. By symmetry of $\bar{A}$, $v_{i} = v^{c}_{i}$ for all $i = 1, \dots, N_{a}$ and the theorem follows from Theorem \ref{thm:agree_disagree}.A.
\end{proof}


We next state an analogous result for disagreement equilibria on a common class of graphs called \textit{bipartite graphs}.
\begin{definition}
Undirected $\hat{G} = (\hat{V}, \hat{E})$ is a {\em bipartite graph} 
if $\hat V$ can be subdivided into disjoint subsets $\hat{V}_{1}$, $\hat{V}_{2}$ such that every edge $\hat{e}_{ik} \in \hat{E}$ connects a vertex in $\hat{V}_{1}$ to one in $\hat{V}_{2}$.  \label{def:bipartite}
\end{definition}

\noindent In the following, we show that for  disagreement on bipartite graphs, an agent's partition membership  determines the sign of its equilibrium opinion and an agent's centrality in the network determines the strength of its equilibrium opinion. 
\begin{lemma}[
$\lambda_{min}$-Eigenvector of Bipartite Graph]
Suppose $\bar{G}$ is a bipartite graph and let $\hat{V}_{1}, \hat{V}_{2}$ be the two vertex subsets of $V$ from Definition \ref{def:bipartite}. Let $\mathbf{w} = (w_{1}, \dots, w_{N_{a}})$ be the eigenvector corresponding to $\lambda_{min}$ of $\bar{A}$. Then $w_{i}, w_{k} \neq 0$ and $\operatorname{sign}(w_{i}) = - \operatorname{sign}(w_{k})$ for all $i \in \hat{V}_{1}$, $k \in \hat{V}_{2}$. Moreover,  $|w_{i}| = v^{c}_{i}$ for all $i = 1, \dots, N_{a}$. \label{lem:diagr_centrality}
\end{lemma}
\begin{proof} For all bipartite graphs, $\operatorname{dim} W(\lambda_{min}) = 1$. By symmetry of $\bar{A}$, the  $\lambda_{max}$-eigenvector is equal to  $\mathbf{v}^{c}$. 
The lemma then follows from \cite[Theorem 1.2]{stevanovic2014spectral}, which states that the terms of the $\lambda_{min}$-eigenvector are equal in magnitude to the terms of the $\lambda_{max}$-eigenvector, with the sign structure reflecting the bipartition.
\end{proof}

\begin{theorem}[Disgreement Opinion Strength Reflects Agent Centrality on Bipartite Graphs]
Consider \eqref{eq:opinion_dynamics_hom} with undirected bipartite $\bar{G}$, $u_{i} := u \geq 0$, $b_{i} = 0$, $\alpha > 0$, and $\gamma < 0$. Let $\hat{V}_{1}, \hat{V}_{2}$ be the two subsets of $V$ from Definition \ref{def:bipartite}. Disagreement equilibria $\mathbf{x} = (x_{1}, \dots, x_{N_{a}})$ described in Theorem \ref{thm:agree_disagree}.B satisfy $|x_{i}| < |x_{k}|$ if $v^{c}_{i} < v^{c}_{k}$ and $|x_{i}| = |x_{k}|$ if $v^{c}_{i} = v^{c}_{k}$ for all $i,k = 1, \dots, N_{a}$. Moreover,  $\operatorname{sign}(x_{i}) = - \operatorname{sign}(x_{k})$ for all $i \in \hat{V}_{1}$, $k \in \hat{V}_{2}$ \label{thm:centrality_dis}.
\end{theorem}
\begin{proof}
This follows by Theorem \ref{thm:agree_disagree}B and Lemma \ref{lem:diagr_centrality}.
\end{proof}

All graphs shown in Figure \ref{fig:GraphClasses} have a simple $\lambda_{min}$. Further, the cycle, path, and star are bipartite graphs and the sign distribution of nodes across options reflects their bipartition. Observe that the magnitude of opinions reflects relative centrality on all the shown graphs in both the agreement and disagreement parameter regimes, including the wheel which is not bipartite. This suggests that the $\lambda_{min}$-eigenvector can sometimes be related to a notion of agent centrality even when it does not precisely equal $\mathbf{v}^{c}$. In contrast, graphs in Figure \ref{fig:MoreGraphs} are not bipartite, and graphs $A,B$ have a two-dimensional $W(\lambda_{min})$. The distribution of disagreement opinions on these graphs is more heterogeneous in magnitude, despite the first two graphs being regular and all agents being equally central. Additional equilibrium distributions of opinions are possible in the disagreement regime. 


\section{Graph Symmetry \label{sec:symmetry}}

In this section we relate the patterns of emergent opinions on solution branches described in Theorem \ref{thm:agree_disagree} to the symmetry of the underlying graph. We first state several key definitions.

Recall that a \textit{graph automorphism} of $\bar{G}$ is a permutation of vertices that preserves adjacency (i.e., maps edges to edges and non-edges to non-edges). The \textit{automorphism group} $\operatorname{Aut}(\bar{G}) := \Gamma$ is the set of all graph automorphisms of $\bar{G}$. Define the group of permutations of a set of $n$ symbols $\{1, \dots, n \}$ as $S_{n}$. Then for graph $\bar{G}$ with $N_{a}$ agent vertices, $\Gamma \subseteq S_{N_{a}}$. The graph automorphism group of $\bar{G}$ can also be interpreted as the group of permutation matrices which commute with its adjacency matrix $\bar{A}$. Each element $\rho \in \Gamma$ can be represented by a permutation matrix $P_{\rho}$ which satisfies $P_{\rho} \bar{A} = \bar{A} P_{\rho}$ \cite[Proposition 3.8.1]{cvetkovic2009introduction}. We commonly refer to $\rho$ as a \textit{symmetry} of $\bar{G}$ and $\Gamma$ as its \textit{symmetry group}. 

A connected notion is the \textit{equivariance} of a dynamical system with respect to a symmetry group. Consider the opinion dynamics \eqref{eq:opinion_dynamics_hom} as a dynamical system $ \dot{\mathbf{x}} = \mathbf{h}(\mathbf{x})$, where the map $\mathbf{h}: \mathds{R}^{N_{a}} \to \mathds{R}^{N_{a}}$ is $\mathbf{h}(\mathbf{x}) = (h_{1}(\mathbf{x}),\dots,h_{N_{a}}(\mathbf{x}) )$. Let $\Sigma$ be a compact Lie group with elements $\sigma$ that act on $\mathds{R}^{N_{a}}$. Then $\mathbf{h}$ is $\sigma$-\textit{equivariant} for some $\sigma \in \Sigma$ if $\sigma \mathbf{h}(\mathbf{x}) = \mathbf{h}(\sigma\mathbf{x})$, and $\mathbf{h}$ is $\Sigma$-\textit{equivariant} if this holds true for all $\sigma\in \Sigma$ \cite[Definition 1.7]{GolubitskySymmetryPerspective}. In the following theorem we show that the symmetry group of the graph $\bar{G}$ is also a symmetry group of the opinion dynamics \eqref{eq:opinion_dynamics_hom}  with zero input.
\begin{theorem}[$\Gamma$-equivariance]
\label{thm:equivariance} Consider opinion dynamics \eqref{eq:opinion_dynamics_hom} with $b_{i} = 0$ for $i = 1, \dots, N_{a}$. Let $\Gamma$ be the automorphism group of $\bar{G}$. The opinion dynamics are $\Gamma$-equivariant. 
\end{theorem}
\begin{proof}
Let $\rho$ be an element of $\Gamma$, and define the function $\hat{\mathbf{S}}: \mathds{R}^{N_{a}} \to \mathds{R}^{N_{a}}$ as $\hat{\mathbf{S}}(\mathbf{y}) = ( S(y_{1}), \dots, S(y_{N_{a}}))$. Every permutation $\rho$ can be decomposed into a product of transpositions $\rho_{ik}$ that interchange elements in positions $i$ and $k$ of a set. For an arbitrary transposition, it is easy to see that $P_{\rho_{ik}}\hat{\mathbf{S}}\big(\mathbf{y} \big) = \hat{\mathbf{S}}\big( P_{\rho_{ik}}\mathbf{y} \big) $, and by iteratively applying transpositions we see that for all permutations $\rho$, $P_{\rho}\hat{\mathbf{S}}\big(\mathbf{y} \big) = \hat{\mathbf{S}}\big( P_{\rho}\mathbf{y} \big)$. Using this and the definition of graph automorphism we get 
\begin{multline*}
    P_{\rho} \mathbf{h}(\mathbf{x}) = - d P_{\rho}  \mathbf{x} + u \hat{\mathbf{S}}\big(P_{\rho} (\alpha  \mathcal{I} + \gamma  \bar{A}) \mathbf{x} \big)\\
    = - d P_{\rho}  \mathbf{x} + u \hat{\mathbf{S}}\big( (\alpha \mathcal{I}  + \gamma \bar{A})  P_{\rho} \mathbf{x} \big) = \mathbf{h}(P_{\rho} \mathbf{x}).
\end{multline*}
This holds for all $\rho \in \Gamma$, and the theorem follows.
\end{proof}


In the remainder of this section we illustrate how graph symmetry dictates patterns of opinions on the network.

\begin{definition} \label{def:orbit_partition}
Let $\Gamma$ be the automorphism group  of $\bar{G}$, and  $i\in V$  a vertex. An {\em orbit of $i$} is
$O_{i} = \{ k \in V \ | k = \rho i \  \text{for some } \rho \in \Gamma \}$. The orbits are equivalence classes that partition $V$ through the equivalence relation 
\begin{equation}
    i \sim k \ \text{if} \ k = \rho i \ \text{for some }\rho \in \Gamma. \label{eq:relation}
\end{equation}
\end{definition}

\begin{theorem} Consider opinion dynamics \eqref{eq:opinion_dynamics_hom} with $u_{i} := u \geq 0$ and  $b_{i} = 0$ for all $i \in V$. Let $\Gamma$ be the automorphism group of the undirected graph $\bar{G}$, and for any two vertices $i,k \in V$ define the equivalence relation $i \sim k$ as in \eqref{eq:relation}. 

\noindent \textbf{A.} Suppose $\gamma > 0$. For the agreement equilibria $\mathbf{x} = (x_{1}, \dots, x_{N_{a}})$ from Theorem \ref{thm:agree_disagree}.A, if $i \sim k$, then $x_{i} = x_{k}$. 

\noindent \textbf{B.} Suppose $\gamma < 0$ and $\lambda_{min}$ has multiplicity 1. For the disagreement equilibria $\mathbf{x} = (x_{1}, \dots, x_{N_{a}})$ from Theorem \ref{thm:agree_disagree}.B, if $i \sim k$, then $|x_{i}| = |x_{k}|$. \label{thm:symmetry}
\end{theorem}
\begin{proof}
In A and B, the solutions $\mathbf{x}$  
appear along the 1-dimensional subspace $\operatorname{ker} J(\mathbf{0})$. For any $\rho \in \Gamma$ and $\mathbf{v} \in \operatorname{ker} J(\mathbf{0})$, $P_{\rho} \mathbf{v} \in \operatorname{ker} J(\mathbf{0})$ \cite[Remark 1.25]{GolubitskySymmetryPerspective},
and $P_{\rho} \mathbf{v} = \pm \mathbf{v}$ \cite[Lemma 3.8.2]{cvetkovic2009introduction}. The only way for this to be true is if for any $i,k \in V$ for which $i \sim k$, $|v_{i}| = |v_{k}|$. 
\end{proof}

All graphs shown in Figure \ref{fig:GraphClasses} have a nontrivial symmetry and a simple $\lambda_{min}$. The vertices contained in the same orbit under the action of its symmetry have equal magnitude opinion in both agreement and disagreement parameter regimes (e.g. all the vertices of a cycle, the outer vertices of the star and wheel). The randomly generated graph in Figure \ref{fig:MoreGraphs}C also has a simple $\lambda_{min}$ but only a trivial symmetry. Each node is its own orbit and the equilibrium opinion magnitude is different at each node. In contrast, both graphs in Figure \ref{fig:MoreGraphs}A,B have a nontrivial symmetry but the dimension of $W(\lambda_{min})$ is 2 and the conditions of Theorem \ref{thm:equivariance} are not met. The disagreement opinion magnitudes on these graphs do not reflect the orbit structure of the graph's symmetry group.



{ \color{red}

}

\section{Agent Sensitivity and Opinion Cascades \label{sec:sensitivity}}


In  \cite[Section VI]{Bizyaeva2020} a state feedback law is introduced for the attention parameter $u_{i}$ in \eqref{eq:opinion_dynamics_hom}, which enables the group to become opinionated in response to  input $b_{i}$. Each agent's attention parameter $u_{i}$  tracks a saturated norm of the system state observed by agent $i$, with dynamics defined by
\begin{equation}
    \tau_s \dot{u}_i = -u_i + S_u  \left( x_{i}^{2} +\sum_{k = 1}^{N_a} ( \bar{a}_{ik} x_k)^2 \right ). \label{eq:udot}
\end{equation}
Here, $\tau_{s}>0$ is the time scale of the integration and
$S_u(y) = u_f (F(g(y - y_m)) - F(-gy_m))$,
with $F(x) = \frac{1}{1 + e^{-x}}$.  Design parameters $g,u_{f},y_{m} > 0$  tune the system response. The positive feedback between opinion strength and attention defined by dynamics~\eqref{eq:udot} provides the resulting opinion-attention dynamics with a threshold for the input strength needed to let an originally weakly opinionated ($|x_i|\ll\vartheta$, for all $i=1,\ldots,\Na$), weakly attentive ($|u_i|\ll\vartheta$, for all $i=1,\ldots,\Na$) network transition to a strongly opinionated ($|x_i|\gg\vartheta$, for all $i=1,\ldots,\Na$), strongly attentive ($|u_i|\gg\vartheta$, for all $i=1,\ldots,\Na$) state~\cite{Franci2021}. 

Such an \textit{opinion cascade} is illustrated in Figures~\ref{fig:BalancedTree} and~\ref{fig:RandomGraph} top left. An opinion cascade is of agreement (disagreement) type if the final opinion configuration is of agreement (disagreement) type.
A rigorous analysis of opinion cascade over networks is the subject of ongoing work based on unfolding theory~\cite{Golubitsky1985} of agreement and disagreement bifurcations. Here, we illustrate with numerical examples. We show in particular how the spectral properties of the adjacency matrix predict which agents are more likely to trigger a cascade when excited by an 
input.

We conjecture that the norm  of the projection of the vector of inputs $\mathbf{b}$ onto the eigenspaces $W(\lambda_{max})$ or $W(\lambda_{min})$, depending on the cooperative or competitive regime, is directly related to whether or not the opinion cascade gets triggered. To illustrate this, in Figures \ref{fig:BalancedTree} and \ref{fig:RandomGraph} we show two examples of a disagreement cascade for graphs with a simple $\lambda_{min}$. In Figure \ref{fig:BalancedTree} the graph is a balanced tree, which is bipartite. Thus, the magnitude of each entry of its $\lambda_{min}$-eigenvector corresponds to the corresponding agent's relative centrality, as proved in Lemma \ref{lem:diagr_centrality}. When an input is given to the most central agent an opinion cascade gets triggered, whereas when the same input is given to an agent located on an outer leaf, the group remains neutral. 

Analogously, in Figure \ref{fig:RandomGraph} a cascade is triggered on a randomly generated network when an input is given to the agent with the largest magnitude entry in the $\lambda_{min}$-eigenvector, and is not triggered when the same input is given to the agent with the smallest entry. The post-cascade pattern of opinions on this network is very close to the pattern for the solution of the zero-input system shown in Figure \ref{fig:MoreGraphs}C. In this way, the spectral properties of $\bar{A}$ can be used not only to predict the patterns of opinion formation for opinion dynamics \eqref{eq:opinion_dynamics_hom} both with a static $u_{i} = u$ and coupled with \eqref{eq:udot}, but also to inform which agent should receive a control signal or carry a sensor in order to have maximal likelihood of an informed opinion cascade spreading on the network.

\begin{figure}[h]
    \centering
    \includegraphics[width = 0.47
    \textwidth]{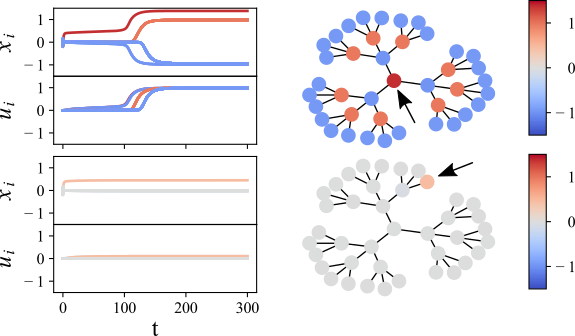}
    \caption{Triggering a cascade on the balanced tree. Input $b_{i} = 0.4$ for agent marked with arrow; $b_{i} = 0$ for all other agents. Simulations of \eqref{eq:opinion_dynamics_hom} start from  small random initial  opinions  drawn from $\mathcal{N}(0,0.1)$.
    Left: time trajectories of the dynamics.  Right: final opinion, indicated by color, of simulation  at $t=300$.   Parameters: $d = 1$, $\alpha = 1$, $\gamma = -1$, $u_{i}(0) = 0$ for all $i \in V$, $g = 10$, $y_{m} = 0.4$, $u_{f} = 1$, $\tau_{s} = 10$. Only in the top plots does the model undergo a disagreement opinion cascade, with $\vartheta=0.5$.
    }
    \label{fig:BalancedTree}
\end{figure}

\begin{figure}[h]
    \centering
    \includegraphics[width = 0.47
    \textwidth]{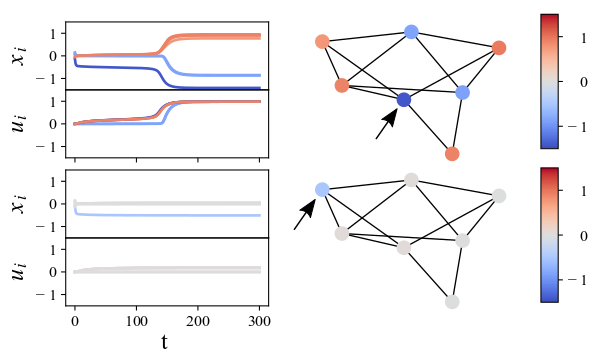}
    \caption{Triggering a cascade on a randomly generated graph. Input $b_{i} = -0.45$ for agent marked with arrow; $b_{i} = 0$ for all other agents.  Simulations of \eqref{eq:opinion_dynamics_hom} start from  small random initial  opinions  drawn from $\mathcal{N}(0,0.1)$.  Left: time trajectories of the dynamics. Right: final opinion, indicated by color, of simulation at $t=300$.  Parameters: $d = 1$, $\alpha = 0.5$, $\gamma = -0.5$, $u_{i}(0) = 0$ for all $i \in V$, $g = 10$, $y_{m} = 0.4$, $u_{f} = 1$, $\tau_{s} = 10$. Only in the top plots does the model undergo a disagreement opinion cascade, with $\vartheta=0.5$. 
    }
    \label{fig:RandomGraph} \vspace*{-5mm}
\end{figure}

\section{Discussion \label{sec:dis}}
We have shown how patterns of nonlinear opinion formation on two-option networks can be studied systematically using spectral properties of the graph adjacency matrix. We proved that agent centrality and symmetry of the underlying graph select the pattern of opinions in agreement and disagreement  for a large set of networks. We illustrated how the eigenvectors of the adjacency matrix inform sensitivity of individual nodes to input that triggers opinion cascades; proving the influence of distributed input on opinion cascades is the subject of ongoing work. Other future work includes the natural generalization to multi-option networks with opinion dynamics defined in \cite{Bizyaeva2020}.

\bibliographystyle{./bibliography/IEEEtran}
\bibliography{./bibliography/references}

\begin{thebibliography}{10}
\providecommand{\url}[1]{#1}
\csname url@samestyle\endcsname
\providecommand{\newblock}{\relax}
\providecommand{\bibinfo}[2]{#2}
\providecommand{\BIBentrySTDinterwordspacing}{\spaceskip=0pt\relax}
\providecommand{\BIBentryALTinterwordstretchfactor}{4}
\providecommand{\BIBentryALTinterwordspacing}{\spaceskip=\fontdimen2\font plus
\BIBentryALTinterwordstretchfactor\fontdimen3\font minus
  \fontdimen4\font\relax}
\providecommand{\BIBforeignlanguage}[2]{{%
\expandafter\ifx\csname l@#1\endcsname\relax
\typeout{** WARNING: IEEEtran.bst: No hyphenation pattern has been}%
\typeout{** loaded for the language `#1'. Using the pattern for}%
\typeout{** the default language instead.}%
\else
\language=\csname l@#1\endcsname
\fi
#2}}
\providecommand{\BIBdecl}{\relax}
\BIBdecl

\bibitem{Bizyaeva2020}
A.~Bizyaeva, A.~Franci, and N.~E. Leonard, ``A general model of opinion
  dynamics with tunable sensitivity,'' \emph{arXiv:2009.04332 [math.OC]}, Sep.
  2020.

\bibitem{Franci2020}
A.~Franci, M.~Golubitsky, A.~Bizyaeva, and N.~E. Leonard, ``A model-independent
  theory of consensus and dissensus decision making,'' \emph{arXiv:1909.05765
  [math.OC]}, Sep. 2020.

\bibitem{OlfatiSaber2004}
R.~{Olfati-Saber} and R.~M. {Murray}, ``Consensus problems in networks of
  agents with switching topology and time-delays,'' \emph{IEEE Transactions on
  Automatic Control}, vol.~49, no.~9, pp. 1520--1533, 2004.

\bibitem{Altafini2013}
C.~Altafini, ``Consensus problems on networks with antagonistic interactions,''
  \emph{IEEE Trans. Aut. Ctrl.}, vol.~58, no.~4, pp. 935--946, 2013.

\bibitem{Liu2017}
J.~{Liu}, X.~{Chen}, T.~{Başar}, and M.~A. {Belabbas}, ``Exponential
  convergence of the discrete- and continuous-time altafini models,''
  \emph{IEEE Trans. Automatic Control}, vol.~62, no.~12, pp. 6168--6182, 2017.

\bibitem{shi2019dynamics}
G.~Shi, C.~Altafini, and J.~S. Baras, ``Dynamics over signed networks,''
  \emph{SIAM Review}, vol.~61, no.~2, pp. 229--257, 2019.

\bibitem{AF-VS-NEL:15a}
A.~Franci, V.~Srivastava, and N.~E. Leonard, ``A realization theory for
  bio-inspired collective decision making,'' \emph{arXiv:1503.08526 [math.OC]},
  Mar. 2015.

\bibitem{fontan2017multiequilibria}
A.~Fontan and C.~Altafini, ``Multiequilibria analysis for a class of collective
  decision-making networked systems,'' \emph{IEEE Transactions on Control of
  Network Systems}, vol.~5, no.~4, pp. 1931--1940, 2017.

\bibitem{Gray2018}
R.~{Gray}, A.~{Franci}, V.~{Srivastava}, and N.~E. {Leonard}, ``Multiagent
  decision-making dynamics inspired by honeybees,'' \emph{IEEE Transactions on
  Control of Network Systems}, vol.~5, no.~2, pp. 793--806, 2018.

\bibitem{Abara2018}
P.~U. {Abara}, F.~{Ticozzi}, and C.~{Altafini}, ``Spectral conditions for
  stability and stabilization of positive equilibria for a class of nonlinear
  cooperative systems,'' \emph{IEEE Transactions on Automatic Control},
  vol.~63, no.~2, pp. 402--417, 2018.

\bibitem{Fontan2018}
A.~{Fontan} and C.~{Altafini}, ``Achieving a decision in antagonistic multi
  agent networks: frustration determines commitment strength*,'' in \emph{IEEE
  Conference on Decision and Control}, 2018, pp. 109--114.

\bibitem{farina2011positive}
L.~Farina and S.~Rinaldi, \emph{Positive linear systems: theory and
  applications}.\hskip 1em plus 0.5em minus 0.4em\relax John Wiley \& Sons,
  2011, vol.~50.

\bibitem{bonacich1972}
P.~Bonacich, ``Factoring and weighting approaches to status scores and clique
  identification,'' \emph{J. Math. Sociol.}, vol.~2, no.~1, pp. 113--120, 1972.

\bibitem{bonacich2007some}
------, ``Some unique properties of eigenvector centrality,'' \emph{Social
  Networks}, vol.~29, no.~4, pp. 555--564, 2007.

\bibitem{stevanovic2014spectral}
D.~Stevanovic, \emph{Spectral radius of graphs}.\hskip 1em plus 0.5em minus
  0.4em\relax Academic Press, 2014.

\bibitem{cvetkovic2009introduction}
D.~Cvetkovic, S.~Simic, and P.~Rowlinson, \emph{An Introduction to the Theory
  of Graph Spectra}.\hskip 1em plus 0.5em minus 0.4em\relax Cambridge
  University Press, 2009.

\bibitem{GolubitskySymmetryPerspective}
M.~Golubitsky and I.~Stewart, \emph{The Symmetry Perspective}, 1st~ed., ser.
  Progress in Mathematics.\hskip 1em plus 0.5em minus 0.4em\relax
  Birkh{\"a}user Basel, 2002, vol. 200.

\bibitem{Franci2021}
A.~Franci, A.~Bizyaeva, S.~Park, and N.~Leonard, ``Analysis and control of
  agreement and disagreement opinion cascades,'' \emph{arXiv}, 2021.

\bibitem{Golubitsky1985}
M.~Golubitsky and D.~G. Schaeffer, \emph{Singularities and Groups in
  Bifurcation Theory}, ser. Applied Mathematical Sciences.\hskip 1em plus 0.5em
  minus 0.4em\relax New York, NY: Springer-Verlag, 1985, vol.~51.

\end{thebibliography}

\end{document}